\documentclass[11pt,twoside, reqno]{amsart}

\usepackage{amsmath}
\usepackage{amsthm}
\usepackage{amsfonts, amssymb}
\usepackage{mathrsfs}
\usepackage[all]{xy}
\usepackage{url}

\usepackage{latexsym}
\usepackage[dvips]{graphicx}

\theoremstyle{plain}
\newtheorem{lema}{Lemma}
\newtheorem{prop}[lema]{Proposition}
\newtheorem{teo}[lema]{Theorem}

\newtheorem{coro}[lema]{Corollary}
\theoremstyle{remark}

\newtheorem{obs}[lema]{Remark}

\theoremstyle{definition}
\newtheorem{defi}[lema]{Definition}

\newcommand{\st}{\textrm{st}}
\newcommand{\lk}{\textrm{lk}}

\newcommand{\dime}{\textrm{dim}}

\newcommand{\qq}{\mathbb{Q}}
\newcommand{\x}{\mathcal{X}}
\newcommand{\kp}{\mathcal{K}}
\newcommand{\kun}{\mathfrak{K}}

\begin{document}

\title[The fixed point property in every weak homotopy type]{The fixed point property in every weak homotopy type}

\author[J.A. Barmak]{Jonathan Ariel Barmak $^{\dagger}$}

\thanks{$^{\dagger}$ Researcher of CONICET. Partially supported by grants UBACyT 20020100300043, CONICET PIP 112-201101-00746 and ANPCyT PICT-2011-0812.}

\address{Departamento de Matematica\\
FCEyN-Universidad de Buenos Aires\\
Buenos Aires, Argentina}

\email{jbarmak@dm.uba.ar}

\begin{abstract}
The Brouwer fixed point theorem states that the disk $D^n$ has the fixed point property. More generally, by the Lefschetz fixed point theorem any compact ANR with trivial rational homology has the fixed point property. In this note we prove that for any connected compact CW-complex $K$ there exists a space $X$ weak homotopy equivalent to $K$ which has the fixed point property. The result is known to be false if we require $X$ to be a polyhedron. The space $X$ we construct is a non-Hausdorff space with finitely many points.   
\end{abstract}



\maketitle

A topological space $X$ has the fixed point property if every continuous map $f:X\to X$ has a fixed point. We will prove the following

\begin{teo} \label{main}
  Let $K$ be a connected compact CW-complex. Then there exists a topological space $X$ weak homotopy equivalent to $K$ with the fixed point property.
 
\end{teo}

If we require $X$ to be a polyhedron, the result is known to be false. Every polyhedron homotopy equivalent to $S^n$ lacks the fixed point property for $n\ge 3$ (see \cite{Lop} or the proof of \cite[Theorem]{Wag} for example). The space $X$ we find has finitely many points. Therefore, we are also proving the following result. The homotopy type of any connected compact CW-complex can be realized by the order complex of a finite partially ordered set with the fixed point property.

A simplicial complex $K$ has the \textit{fixed simplex property} if for every simplicial map $f:K\to K$ there exists a simplex $\sigma \in K$ such that $f(\sigma)=\sigma$ or, equivalently, if every simplicial endomorphism of $K$ fixes a point of the realization of $K$. The spheres $S^n$ do not have the fixed point property, but they do have triangulations with the fixed simplex property provided that $n\ge 2$ (see Proposition \ref{subdivasim}). We will show that for every simply connected compact polyhedron $K$ there exists a finite simplicial complex $L$ homotopy equivalent to $K$ with the fixed simplex property. Then the finite topological space $\x (L)$ associated to $L$ has the fixed point property. This will prove Theorem \ref{main} for simply connected complexes. If $K$ is not simply connected we will be able to modify the construction above to obtain a finite model of $K$ with the fixed point property but it will not be the poset of faces of a complex.

We sketch in a few lines the idea of the construction of $L$ from $K$ and the main parts of the proof. We first consider integer homology classes $\alpha _{k,l}\in H_k(K)$ which are a basis of the rational $k$-homology of $K$ and then realize each $\alpha _{k,l}$ as the image of the fundamental class $[M_{k,l}]$ of a $k$-dimensional oriented pseudomanifold through a map $M_{k,l}\to K$. We construct $L$ as follows. We find a sufficiently fine ($j$-th barycentric) subdivision $K^j$ of $K$ and attach to $K^j$ mapping cylinders of the maps $M_{k,l}^j\to K^j$ and of approximations to the identities $M_{k,l}^j\to M_{k,l}^r$ ($r<j$). In this way we manage to concentrate the homology classes $\alpha _{k,l}$ in ``few'' simplices, those of $M_{k,l}^r$. The complex $L$ satisfies this singular property: if $c=\sum t_i \sigma_i \in C_k(L)$ is a cycle such that $\sum |t_i|$ is less than or equal to the number of $k$-simplices in $M_{k,l}^r$ and the $(\alpha _{k,l}\otimes 1_{\qq})$-coordinate of $[c]\in H_k(L; \qq)$ is different from zero, then $c$ is, up to sign, the fundamental cycle of $M_{k,l}^r$. Therefore if $f$ is a simplicial endomorphism of $L$, one of the following holds: (1) the matrices of $f_*:H_k(L; \qq) \to H_k(L; \qq)$ in the basis $\{\alpha _{k,l}\otimes 1_{\qq}\}_l$ have all the diagonal zero for each $k\ge 2$, in which case the Lefschetz fixed point theorem gives us a fixed simplex, or (2) the map $f$ maps one of the $M_{k,l}^r$ into itself and $f|_{M_{k,l}^r}: M_{k,l}^r \to M_{k,l}^r$ is a simplicial automorphism. The argument is then complete if we show that the pseudomanifolds $M_{k,l}^r$ can be assumed to be \textit{asymmetric}, in the sense that every automorphism fixes a vertex.                   


\medskip

The following notions are a rigid version of Gromov's simplicial volume and the $\ell^1$-norm. Let $C$ be a finitely generated free $\mathbb{Z}$-module with a fixed basis $\{b_1, b_2, \ldots , b_r\}$. The \textit{norm} $\| c \|$ of an element $c=\sum{t_i b_i}\in C$ is $\sum{|t_i|}$. If $f:C\to C'$ is a morphism between finitely generated free $\mathbb{Z}$-modules (each of them with a chosen basis), the \textit{norm} of $f$ is $\| f \|= \max\limits_{c\neq 0} \frac{\| f(c) \|}{\|c\|} $. Note that $\|f\|$ is well-defined since $\|f(c)\|\le \|c\| \max\limits_i \|f(b_i)\|$, where $\{b_i\}_i$ is the chosen basis of $C$. If $C=0$ define $\|f\|=0$. For a composition $fg$ we have $\|fg\|\le \|f\| \|g\|$.

Let $C_*$ be a finitely generated free chain complex with a given basis for each $C_k$. The \textit{norm} of $\alpha \in H_k(C)$ is $\| \alpha \|=\min \{ \|c\| \ | \ [c]=\alpha \}$. Here $[c]$ denotes the class of a cycle $c$ in homology.


When $K$ is a finite simplicial complex we will always consider the chain complex $C_*(K)$ with the usual basis for $C_k(K)$ given by one oriented $k$-simplex $[v_0,v_1, \ldots ,v_k]$ for each $k$-simplex $\{v_0, v_1, \ldots, v_k\} \in K$. We denote by $H_*(K)$ the simplicial homology of $K$ with integer coefficients.

If $M$ is a closed $n$-dimensional oriented pseudomanifold, the norm $\|[M]\|\in H_n(M)$ of its fundamental class is the number of $n$-simplices in $M$. 

If $\varphi :K \to L$ is a simplicial map between finite simplicial complexes, $\varphi _* :C_*(K)\to C_*(L)$ maps an oriented $k$-simplex $[v_0, v_1, \ldots , v_k]$ to $[\varphi(v_0), \varphi(v_1), \ldots , \varphi(v_k)]$ if $\varphi(v_i)\neq \varphi(v_j)$ for $i\neq j$ and to $0$ otherwise. Therefore $\varphi _k:C_k(K) \to C_k(L)$ has norm at most $1$.


If $L$ is a finite simplicial complex and $K$ is a subdivision of $L$, the subdivision operator $\lambda : C_* (L) \to C_*(K)$ is a homotopy inverse to the chain map induced by any simplicial approximation to the identity and maps a $k$-simplex $\sigma \in L$ into a signed sum of all the $k$-simplices of $K$ contained in $L$. Therefore, the norm of $\lambda _k :C_k(L) \to C_k(K)$ is the maximum number of $k$-simplices in which a $k$-simplex of $L$ is subdivided. In particular, when $K=L'$ is the first barycentric subdivision of $L$, $\| \lambda _k \|=(k+1)!$ if $\dim L\ge k$. In this case, if $\alpha \in H_k(L)$, $\|\lambda _* (\alpha)\|\le \| \lambda _k\| \| \alpha \| \le (k+1)! \| \alpha \|$. In general the equality $\| \lambda _* (\alpha) \|=(k+1)! \| \alpha \|$ does not hold if $k<\dim (L)$.   

Let $K$ be a finite simplicial complex. The barycenter of a simplex $\sigma \in K$ will be denoted by $b(\sigma)$ or $\hat{\sigma}$. The simplices of $K'$ are then the sets $\{ \hat{\sigma} _0 , \hat{\sigma}_1, \ldots , \hat{\sigma}_k \}$ such that $\sigma _i \subsetneq \sigma _{i+1}$ for every $i$. Given a simplicial map $\varphi :K\to L$, we denote by $\varphi ' :K' \to L'$ the map $b(\sigma) \mapsto b(\varphi (\sigma))$ and in general $\varphi ^j :K^j \to L^j$ is the map induced in the $j$-th barycentric subdivisions.

In \cite[Example 2.4]{BB}, Baclawski and Bj\"orner construct a finite model of $S^2$ with the fixed point property. The following definition is inspired by their example. 

\begin{defi}
We will say that a complex $K$ is \textit{asymmetric} if there exists a vertex $v\in K$ which is fixed by every simplicial automorphism of $K$.
\end{defi}

\begin{prop} \label{subdivasim}
Let $M$ be an $n$-dimensional pseudomanifold with $n\ge 2$. Then there exists a subdivision $L$ of $M$ such that the $j$-th barycentric subdivision $L^j$ of $L$ is asymmetric for every $j\ge 0$. 
\end{prop}
\begin{proof}
Given a complex $K$ and a simplex $\sigma \in K$, we denote by $\deg _K(\sigma)$ the number of maximal simplices of $K$ containing $\sigma$. Then $\deg _K(\sigma)$ is the number of maximal simplices in the link $\lk _K(\sigma)$ if $\sigma$ is not maximal, and $1$ if $\sigma$ is maximal in $K$. Define $d(K)=\max\limits_{\sigma \in K} \deg _K(\sigma)= \max\limits_{v \in K} \deg _K(v) $ where the second maximum is taken over all the vertices $v$ of $K$. It is not hard to see that there exists a subdivision $L$ of $M$ which contains a vertex $v_0$ such that $\deg _L(v)<\deg _L(v_0)$ for any other $v\in L$. Then clearly $L$ is asymmetric since the degree $\deg$ is preserved by automorphisms of $L$ and thus $v_0$ is fixed by any such an automorphism. Moreover, we will show that $\deg _{L'}(v)<\deg _{L'}(v_0)$ for every vertex $v_0\neq v \in L'$. It follows by induction that $L^j$ is asymmetric for each $j\ge 0$.

Let $v\in L'$, $v=b(\sigma)$ where $\sigma$ is a simplex of $L$. Let $k=\dim \sigma$. A maximal simplex of $\lk _{L'}(b(\sigma))$ is obtained by choosing a chain $\sigma^0 < \sigma^1< \ldots < \sigma ^{k-1}$ of proper faces of $\sigma$ and a chain $\sigma^{k+1} < \sigma^{k+2}< \ldots < \sigma ^{n}$ of simplices containing $\sigma$. There are $(k+1)$ possible choices for $\sigma ^{k-1}$, $k$ for $\sigma^{k-2}$, $\ldots$, $2$ for $\sigma ^0$. On the other hand there are $\deg_L (\sigma)$ choices for $\sigma ^n$, $(n-k)$ for $\sigma ^{n-1}$, $(n-k-1)$ for $\sigma ^{n-2}$, $\ldots$, $2$ for $\sigma ^{k+1}$. Therefore
$$\deg _{L'}(b(\sigma))=(k+1)!(n-k)!\deg _L(\sigma).$$

\noindent -If $k=n$, $\deg _{L'}(b(\sigma))=(n+1)!<n!d(L)$ since $d(L)=n+1$ only when the pseudomanifold is isomorphic to the boundary of an $(n+1)$-simplex, which is not the case since $L$ is asymmetric.

\noindent -If $k=n-1$, $\deg _{L'}(b(\sigma))=n!\deg_L(\sigma)=2n!<n!d(L)$.

\noindent -If $1\le k\le n-2$, ${n \choose k}\ge n$, so $\deg _{L'}(b(\sigma))\le (n-1)!(k+1)\deg _L(\sigma)< n!d(L)$.


\noindent -If $k=0$, $\sigma=v$ is a vertex of $L$ and  $\deg _{L'}(v)=n!\deg _{L}(v)$, which is strictly smaller than $n!d(L)$ if $v\neq v_0$ and which is $n!\deg _L(v_0)=n!d(L)$ if $v=v_0$. Then $d(L')=n!d(L)$ and this degree is only achieved by $v_0$.

\end{proof}

\begin{lema} \label{subcomplejo}
Let $\tau$ be a simplex. Let $\sigma_0, \sigma_1, \ldots , \sigma_k$ be faces of $\tau$ such that for every $0\le i,j \le k$ one of the following holds:
\begin{enumerate}
\item $\sigma _i\subseteq \sigma _j$.
\item $\sigma _j \subseteq \sigma _i$.
\item $\sigma _i \cap \sigma _j =\emptyset$.
\end{enumerate}
Then the convex hull of $ \{\hat{\sigma}_0, \hat{\sigma}_1, \ldots , \hat{\sigma}_k\}$ is a subcomplex of $\tau '$.
\end{lema}
\begin{proof}
We proceed by induction in the number of pairs $i,j$ satisfying (3). If (1) or (2) holds for every $i,j$ then the convex hull $S$ of $ \{\hat{\sigma}_0, \hat{\sigma}_1, \ldots , \hat{\sigma}_k\}$ is a simplex of $K'$. Otherwise take $i,j$ such that $\sigma= \sigma _i \cup \sigma _j$ has maximum cardinality among all the pairs satisfying (3). Since $\hat{\sigma}$ is a convex combination of $\hat{\sigma}_i$ and $\hat{\sigma}_j$, by induction it suffices to prove that $\{ \sigma, \sigma_0, \sigma_1, \ldots , \sigma_{i-1}, \sigma_{i+1}, \ldots, \sigma_k\}$ and $\{ \sigma, \sigma_0, \sigma_1, \ldots , \sigma_{j-1}, \sigma_{j+1}, \ldots, \sigma_k\}$ satisfy the hypothesis of the lemma. Let $0\le l \le k$, $i\neq l \neq j$. We have to verify that $\sigma$ and $\sigma _l$ are comparable or disjoint. If $(i,l)$ and $(j,l)$ satisfy (3), then $\sigma$ and $\sigma _l$ are disjoint. Suppose $(i,l)$ satisfies (3) and $\sigma _j$ is comparable with $\sigma _l$. By the choice of $i$ and $j$, $\sigma _j$ cannot be a proper 
face of $\sigma _l$. Then $\sigma _l \subseteq \sigma _j \subseteq \sigma$ so $\sigma _l$ and $\sigma$ are comparable. By symmetry it only remains to analyze the case that $\sigma _l$ is comparable with both $\sigma _i$ and $\sigma _j$. If $\sigma _l$ is a face of any of them, then it is a face of $\sigma$. If $\sigma _i$ and $\sigma _j$ are faces of $\sigma _l$, then so is $\sigma$. 
\end{proof}

\begin{obs} \label{kmasuno}
In the conditions of Lemma \ref{subcomplejo}, note that if the convex hull $S$ of $ \{\hat{\sigma}_0, \hat{\sigma}_1, \ldots , \hat{\sigma}_k\}$ is a $k$-dimensional subcomplex of $\tau '$, then it contains at most $(k+1)!$ many $k$-simplices of $\tau '$. Moreover, if the equality holds then all the pairs $i,j$ satisfy condition (3). The proof of Lemma \ref{subcomplejo} shows that the vertices of $S\le \tau'$ are barycenters of unions of simplices $\sigma _i$. Therefore, the $k$-simplices of $\tau '$ contained in $S$ are of the form $\{\hat{\tau} _0, \hat{\tau} _1, \ldots , \hat{\tau} _k\}$ where $\tau _l=\bigcup\limits_{i\in A_l} \sigma _i$, $A_l\subseteq [0,k]$, and $\tau _l \subsetneq \tau_{l+1}$. Moreover, we can assume that $A_l\subsetneq A_{l+1}$. Since there are only $(k+1)!$ sequences $\emptyset \neq A_0\subsetneq A_1 \subsetneq \ldots \subsetneq A_k \subseteq [0,k]$, $S$ is decomposed in at most $(k+1)!$ many $k$-simplices. If $\sigma _i \subseteq \sigma _j$ for some $i\neq j$ then half of these sequences do not give a $k$-simplex, so $S$ is decomposed in at most $\frac{(k+1)!}{2}$ many $k$-simplices.
\end{obs}

Let $\varphi : K \to L$ be a simplicial map between finite simplicial complexes. We will work with the following version of the simplicial mapping cylinder $Z_{\varphi}$ of $\varphi$. First we choose a total ordering in the set of vertices of $K$. The vertex set of $Z_{\varphi}$ is the disjoint union of the vertex set of $K$ and of $L$. The simplices of the cylinder are the simplices of $L$ together with sets of the form $\{v_0, v_1, \ldots , v_l, \varphi (v_{l+1} ), \varphi (v_{l+2} ), \ldots , \varphi (v_m )\}$ where $\{v_0, v_1, \ldots , v_m\}$ is a simplex of $K$ and $v_i \le v_{i+1}$ for every $i$ ($v_l$ could be equal to $v_{l+1}$).

There is a simplicial retraction $p:Z_{\varphi} \to L$ of the canonical inclusion $j: L\to Z_{\varphi}$ defined by $p(v)=\varphi (v)$ if $v\in K$. Therefore $pi=\varphi$ where $i$ denotes the canonical inclusion of $K$ into the cylinder. The composition $jp$ lies in the same contiguity class as the identity $1_{Z_{\varphi}}$, so $p_* : C_*(Z_{\varphi}) \to C_*(L)$ is a homotopy equivalence \cite[p.151]{Spa}.

Suppose $K$ is a subdivision of a complex $L$ and that $\psi : K \to L$ is a simplicial approximation to the identity. In other words, $\psi$ is a vertex map which maps each vertex $v\in K$ to any vertex $w\in L$ of the unique open simplex of $L$ containing $v$. In this case $\psi _*:C_*(K) \to C_*(L) $ is a homotopy equivalence. Since $p_*:C_*(Z_{\psi}) \to C_*(L)$ is a homotopy equivalence and $\psi =pi$, $i_* : C_*(K) \to C_*(Z_{\psi})$ is a homotopy equivalence. Since $C_*(K)$ is a subcomplex of $C_*(Z_{\psi})$, it is known that there exists a retraction $r: C_*(Z_{\psi}) \to C_* (K)$. However, we need to control the norm $\|r_k \|$ of each $r_k : C_k(Z_{\psi}) \to C_k (K)$. We will prove that for barycentric subdivisions $K=L'$ there is a retraction $r$ such that $\|r_k\|\le (k+1)!$. It is not true that this inequality holds for any retraction $r$.

\begin{lema} \label{ellema}
Let $K$ be a finite simplicial complex. Then there exists an ordering of the vertices of $K'$, a simplicial approximation to the identity $\psi :K' \to K$ and a retraction $r:C_*(Z_{\psi})\to C_*(K')$ satisfying the following
\begin{enumerate}
 \item If $S$ is a $k$-simplex of $Z_{\psi}$, then $\|r_k(S)\|\le (k+1)!$.
 \item If $S$ is a $k$-simplex of $Z_{\psi}$ with $k\ge 1$ such that $\|r_k(S)\|= (k+1)!$, then $S\in K$.
\end{enumerate}
\end{lema}
\begin{proof}
Order the vertices $\hat{\sigma}$ of $K'$ in such a way that $\hat{\sigma}<\hat{\tau}$ implies $\dime (\sigma) \ge \dime (\tau)$. Let $\psi :K' \to K$ be any approximation to the identity. In other words, if $\hat{\sigma}$ is a vertex of $K'$, then $\psi (\hat{\sigma})\in \sigma$.
A $k$-simplex of $Z_{\psi}$ is of the form $S=\{ \hat{\sigma}_0, \hat{\sigma}_1, \ldots , \hat{\sigma}_l, \psi (\hat{\sigma}_{l+1}), \psi (\hat{\sigma}_{l+2}), \ldots ,$ $\psi (\hat{\sigma}_m) \}$ where $l\ge -1$, $m\ge l$ and $\sigma_{i+1} \subseteq \sigma _{i}$ for all $0\le i\le m-1$ (the simplices of $K$ are included in these). We can consider $S$ as a set of vertices of $K'$ identifying $v_i=\psi (\hat{\sigma}_i)$ with the barycenter of $\{v_i\}$. The hypothesis of Lemma \ref{subcomplejo} is satisfied. For any $0\le i,j \le l$, $\sigma _i$ and $\sigma _j$ are comparable; if $l+1\le i,j\le m$, $\{v_i\}$ and $\{v_j\}$ are disjoint or equal; if $0\le i\le l$ and $l+1\le j\le m$, then $\sigma _i \supseteq \sigma _j \supseteq \{v_j\}$. Thus, by Lemma \ref{subcomplejo}, the convex hull of $S$ is a subcomplex $\Phi (S)$ of $K'$.
The application $\Phi$ defines an acyclic carrier from $Z_{\psi}$ to $K'$. Let $r:C_*(Z_{\psi})\to C_*(K')$ be a chain map carried by $\Phi$. Note that $r$ is univocally determined by $\Phi$ since for a $k$-simplex $S\in Z_{\psi}$, $\Phi(S)$ is $j$-dimensional with $j\le k$. Thus, any homotopy $F:C_*(Z_{\psi})\to C_{*+1}(K')$ carried by $\Phi$, given by the Acyclic carrier theorem, must be trivial.

If $S\in Z_{\psi}$ is a $k$-simplex such that $\dim \Phi(S)<k$, then $r(S)\in C_k(\Phi(S))=0$ is trivial. Otherwise $\dim \Phi (S)=k$ and then $\Phi(S)$ is a subdivision of $S$ (considered as a set of $k+1$ affinely independent vertices of $K'$). One has then the subdivision operator $\lambda :C_*(S)\to C_*(\Phi(S))$. Since for each $j$-face $\widetilde{S}$ of $S$, $\Phi(\widetilde{S})$ is a $j$-dimensional subcomplex, the acyclic carrier $\Phi$ when restricted to $C_*(S)$ is the usual subdivision carrier $\Phi(\widetilde{S})=K'(\widetilde{S})$. Thus $\lambda, r|_{C_*(S)}:C_*(S)\to C_*(\Phi(S))$ are carried by the same acyclic carrier $\Phi$ and by the same argument as before, they coincide. Hence $\|r_k(S)\|= \|\lambda _k(S)\|$ is the number of $k$-simplices in $\Phi(S)$ which is at most $(k+1)!$ by Remark \ref{kmasuno}.

Cearly $r:C_*(Z_{\psi})\to C_*(K)$ is a retraction since for $S\in K'$ we have $\Phi(S)=S$ and $\lambda(S)=S$.

Finally, suppose $S=\{ \hat{\sigma}_0, \hat{\sigma}_1, \ldots , \hat{\sigma}_l, \psi (\hat{\sigma}_{l+1}), \psi (\hat{\sigma}_{l+2}), \ldots , \psi (\hat{\sigma}_m) \}$ is a $k$-simplex of $Z_{\psi}$ with $k\ge 1$. If $l\ge 0$, $\sigma _0$ is comparable with each $\sigma_j$ and each $\{v _j\}$. By Remark \ref{kmasuno}, $S$ is subdivided in less than $(k+1)!$ $k$-simplices of $K'$ so $\|r_k (S)\|< (k+1)!$. This proves the second assertion of the lemma. 
\end{proof}

\begin{obs} \label{realizar}
It is well-known that every singular $k$-homology class $\alpha$ of a space $X$ can be realized by a disjoint union $\sqcup M_i$ of closed $k$-dimensional oriented pseudomanifolds, meaning that there is a continuous map $f: \sqcup M_i \to X$ such that $\sum f_*([M_i])=\alpha$ (see \cite[p.108]{Hat} for example). Then for every simplicial homology class $\alpha \in H_k(K)$ of a simplicial complex $K$ there exists a simplicial map $\varphi :\sqcup M_i \to K$ from a disjoint union of closed oriented pseudomanifolds such that $\sum \varphi _* ([M_i])=\alpha$.  
\end{obs} 

%

\begin{teo} \label{fixedsimplex}
Let $K$ be a finite simplicial complex which is simply connected or, more generally, such that $H_1(K)=0$. Then there exists a finite simplicial complex $L$ homotopy equivalent to $K$ with the fixed simplex property. 
\end{teo}
\begin{proof}

\textbf{First part: Construction of $L$}

Let $n=\dim (K)$. For each $2\le k\le n$ let $d_k$ be the rank of $H_k(K; \qq)$. Take for each $2\le k\le n$ homology classes $\alpha _{k,1}, \alpha_{k,2}, \ldots , \alpha_{k, d_k}\in H_k(K)$ such that $\{\alpha_{k,1}\otimes 1_{\qq}, \alpha_{k,2} \otimes 1_{\qq}, \ldots , \alpha_{k,d_k}\otimes 1_{\qq}\}$ is a basis of $H_k(K; \qq)$. By Remark \ref{realizar} each $\alpha _{k,l}$ can be realized by a disjoint union of closed oriented pseudomanifolds. Moreover, by changing the $\alpha _{k,l}$'s if needed we can assume that each of them is realized by a single pseudomanifold. For each $k\ge 2$ and $1\le l\le d_k$ let $M_{k,l}$ be a $k$-dimensional oriented pseudomanifold and let $\varphi _{k,l}: M_{k,l} \to K$ be a simplicial map such that $(\varphi _{k,l})_*([M_{k,l}])= \alpha _{k,l}$. By Proposition \ref{subdivasim} we can assume that $M_{k,l}$ is asymmetric for each $k,l$ as well as all their iterated barycentric subdivisions.

We define an increasing sequence $s_1,s_2,s_3,\ldots ,s_n$ of non-negative integers as follows. Let $s_1=0$. Let $N_{2,l}=\| [M_{2,l}] \|$ be the number of $2$-simplices in $M_{2,l}$ for 
each $1\le l\le d_2$ and let $N_2=\max\limits_l N_{2,l} $. If $P$ is any finite simplicial complex, the cover 
$\mathcal{U}$ of $P$ given by the open stars $\st _P(v)$ of the vertices of 
$P$ has a Lebesgue number $\delta >0$. Therefore, there exists a positive  
integer $s_2$ such that for each $s\ge s_2$, every connected subcomplex of 
$P^{s}$ generated by at most $N_2$ many simplices is contained in an element of 
$\mathcal{U}$, and in particular in a contractible subcomplex of $P$. We take $s_2$ in such a way that the assertion above holds for $P$ when $P$ is any $M_{k,l}$ with $k\ge 3$ and $1\le l\le d_k$.  

Now let $N_{3,l}=\|[M_{3,l}^{s_2}]\|$ for each $1\le l\le d_3$ and let $N_3=\max\limits_l N_{3,l}$. Take $s_3\ge s_2$ 
such that for each $s\ge s_3$, $k\ge 4$ and $1\le l\le d_k$, every connected subcomplex of $M_{k,l}^s$ generated by at most $N_3$ many simplices is contained in a contractible subcomplex. 

In general, suppose $s_2,s_3, \ldots ,s_m$ are defined, with $m\le n-2$. Then define $N_{m+1,l}=\|[M_{m+1,l}^{s_m}]\|$ for each $1\le l\le d_{m+1}$ and let $N_{m+1}=\max\limits_l N_{m+1,l}$. Take $s_{m+1}\ge s_{m}$ such that for each $s\ge s_{m+1}$, every connected subcomplex of $M_{k,l}^s$ generated by at most $N_{m+1}$ many simplices is contained in a contractible subcomplex for each $k\ge m+2$ and $1\le l\le d_k$.

Finally define $N_{n,l}=\|[M_{n,l}^{s_{n-1}}]\|$, $N_n=\max\limits_l N_{n,l}$, $N=\max\limits_{2\le k\le n} N_k$ and take $s_n\ge s_{n-1}$ such that for each $s\ge s_n$, any connected subcomplex of $K^s$ generated by at most $N$ simplices is contained in a contractible subcomplex.

We now define for each $k\ge 2$ and $1\le l\le d_k$ a cylinder $C_{k,l}$ which will be attached to $K^{s_n}$. Each $C_{k,l}$ consists of three parts. The first one is $C_{k,l}^a=Z_{\varphi _{k,l}^{s_n}}$, the cylinder of $\varphi _{k,l}^{s_n}: M_{k,l}^{s_n} \to K^{s_n}$ (see Figure \ref{figura}). The second part $C_{k,l}^b$ is constructed as follows. We glue $N$ cylinders $Z_{1_{M_{k,l}^{s_n}}}$ of the identity $1_{M_{k,l}^{s_n}}: M_{k,l}^{s_n} \to M_{k,l}^{s_n}$, the second base of one with the first base of the following, to build a long cylinder $C_{k,l}^b$ with both bases equal to $M_{k,l}^{s_n}$. The last part $C_{k,l}^{c}$ is the union of $s_n-s_{k-1}$ mapping cylinders. For each $s_{k-1} <m\le s_n$ there is a simplicial approximation to the identity $\psi _{k,l,m}:M_{k,l}^m\to M_{k,l}^{m-1}$ and a retraction $R_m=R_{k,l,m}:C_*(Z_{\psi_{k,l,m}})\to C_*(M_{k,l}^m)$ satisfying properties (1) and (2) in the statement of Lemma \ref{ellema}. When we glue the cylinders $Z_{\psi _{k,l,m}}$ identifying a base of one with a base of the following we obtain a cylinder $C_{k,l}^{c}$ with one base equal to $M_{k,l}^{s_n}$ and the other equal to $M_{k,l}^{s_{k-1}}$. Finally we glue $C_{k,l}^a$ with one extreme of $C_{k,l}^b$ and $C_{k,l}^c$ with the other extreme of $C_{k,l}^b$. This is $C_{k,l}$.

\begin{figure}[h] 
\begin{center}
\includegraphics[scale=0.14]{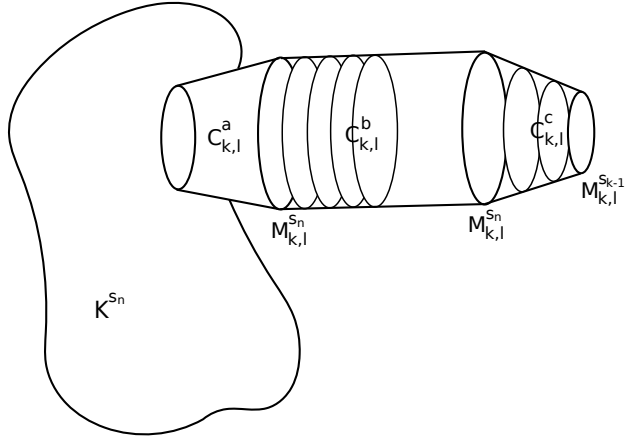}
\caption{The cylinder $C_{k,l}$, union of $C_{k,l}^a$, $C_{k,l}^b$ and $C_{k,l}^c$.}\label{figura}
\end{center}
\end{figure}

Let $L$ be the union of all the cylinders $C_{k,l}$ for $k\ge 2$ and $1\le l\le d_k$, all intersecting in $K^{s_n}$. Each $C_{k,l}^c$ deformation retracts to $M_{k,l}^{s_n}$, $C_{k,l}^b$ deformation retracts to $M_{k,l}^{s_n}$ and $C_{k,l}^a$ deformation retracts to $K^{s_n}$. Therefore $L$ is homotopy equivalent to $K$. We will show that $L$ has the fixed simplex property.

\textbf{Second part: $L$ has the fixed simplex property}

Let $i_{k,l}:M_{k,l}^{s_{k-1}} \hookrightarrow L$ be the inclusion map in the free extreme of $C_{k,l}^c$. Note that $(i_{k,l})_*[M_{k,l}^{s_{k-1}}]=i_*(\varphi _{k,l}^{s_n})_*[M_{k,l}^{s_n}]=i_* \lambda _* (\alpha _{k,l})$, where $i:K^{s_n}\hookrightarrow L$ is the inclusion. Hence, $\mathcal{B}_k=\{ (i_{k,l})_*([M_{k,l}^{s_{k-1}}])\otimes 1_{\qq} \}_l $ is a basis of $H_k(L; \qq)$.

Let $f:L\to L$ be a simplicial map. We study for each $k$ the matrix of $f_*:H_k(L; \qq) \to H_k(L; \qq)$ in the basis $\mathcal{B}_k$. Since $M_{k,l}^{s_{k-1}}$ has $N_{k,l}\le N_k$ many $k$-simplices, $f(M_{k,l}^{s_{k-1}})$ lies in a connected subcomplex of $L$ generated by at most $N_k$ simplices. Since each $C_{k',l'}^b$ is constructed gluing $N\ge N_k$ cylinders then one of the following holds: 

(1) $f(M_{k,l}^{s_{k-1}}) \subseteq \bigcup\limits_{k',l'} (C_{k',l'}^a \cup C_{k',l'}^b)$ or

(2) $f(M_{k,l}^{s_{k-1}}) \subseteq C_{k',l'}^b \cup C_{k',l'}^c$ for some $k',l'$.

In the first case, call $C^{ab}=\bigcup\limits_{k',l'} (C_{k',l'}^a \cup C_{k',l'}^b)$. Just as $L$, the complex $C^{ab}$ deformation retracts to $K^{s_n}$, but in contrast to $L$, for $C^{ab}$ the retraction $r^{ab}:C^{ab} \to K^{s_n}$ may be taken simplicial. Since we are assuming $fi_{k,l}: M_{k,l}^{s_{k-1}}\to C^{ab}$, $r^{ab}fi_{k,l}(M_{k,l}^{s_{k-1}})$ is contained in a connected subcomplex of $K^{s_n}$ generated by $N_{k,l}\le N_k$ simplices. By the choice of $s_n$, this connected subcomplex lies in a contractible subcomplex of $K^{s_n}$ and then $r^{ab}_*(fi_{k,l})_*[M_{k,l}^{s_{k-1}}]=0\in H_k(K^{s_n})$. Since $r^{ab}$ is a homotopy equivalence $(fi_{k,l})_*[M_{k,l}^{s_{k-1}}]=0\in H_k(C^{ab})$ and then $f_*(i_{k,l})_*[M_{k,l}^{s_{k-1}}]=0\in H_k(L)$. In this case the $l$-th column of the matrix of $f_*$ is zero.

Assume then that (2) holds. Call $C_{k',l'}^{bc}=C_{k',l'}^b \cup C_{k',l'}^c$ and suppose that $fi_{k,l}(M_{k,l}^{s_{k-1}})\subseteq C_{k',l'}^{bc}$. The complex $C_{k',l'}^{bc}$ deformation retracts to $M_{k',l'}^{s_{k'-1}}$ by a simplicial retraction $r=r_{k',l'}^{bc}: C_{k',l'}^{bc}\to M_{k',l'}^{s_{k'-1}}$. Since $rfi_{k,l}(M_{k,l}^{s_{k-1}})$ is contained in a connected subcomplex of $M_{k',l'}^{s_{k'-1}}$ generated by $N_{k,l}\le N_k$ simplices, $r_*(fi_{k,l})_*[M_{k,l}^{s_{k-1}}]=0\in H_k(M_{k',l'}^{s_{k'-1}})$ if $k'>k$, by the choices of $s_{k}$ and $s_{k'-1}$. If $k'<k$, then $H_k(M_{k',l'}^{s_{k'-1}})=0$ since $\dim M_{k',l'}=k'$. Therefore in this case we also have $r_*(fi_{k,l})_*[M_{k,l}^{s_{k-1}}]=0\in H_k(M_{k',l'}^{s_{k'-1}})$. Since $r$ is a homotopy equivalence we conclude that if $k'\neq k$, then $(fi_{k,l})_*[M_{k,l}^{s_{k-1}}]=0\in H_k(C_{k',l'}^{bc})$ and therefore $f_*(i_{k,l})_*[M_{k,l}^{s_{k-1}}]=0\in H_k(L)$. Hence the $l$-th column of $f_*$ is also zero. It remains to analyze the case $k'=k$. In that case $r_*(fi_{k,l})_*[M_{k,l}^{s_{k-1}}]\in H_k(M_{k,l'}^{s_{k-1}})$ is an integer multiple of the fundamental class $[M_{k,l'}^{s_{k-1}}]$ and then $f_*(i_{k,l})_*[M_{k,l}^{s_{k-1}}]=(i_{k,l'})_*r_*(fi_{k,l})_*[M_{k,l}^{s_{k-1}}]$ is an integer multiple of $(i_{k,l'})_*[M_{k,l'}^{s_{k-1}}]$. If $l'\neq l$, the $l$-th column of $f_*$ has a zero in the $l$-th entry. The last and most important case is $l'=l$. 

If for each $k,l$ the case (1) occurs or the case (2) for $(k',l')\neq (k,l)$, then the trace of the matrix of $f_*$ in each positive degree is zero and by the Lefschetz fixed point theorem, $f$ must fix a simplex. We can assume then that there exists one pair $k,l$ such that (2) holds for $(k',l')=(k,l)$ and that $ (fi_{k,l})_*[M_{k,l}^{s_{k-1}}]\in H_k(C_{k,l}^{bc}) $ is non-trivial.

The simplicial projection $C_{k,l}^b \to M_{k,l}^{s_n}$ of the cylinder into the extreme in contact with $C_{k,l}^{c}$ extends to a simplicial retraction $p:C_{k,l}^{bc} \to C_{k,l}^c$. On the other hand Lemma \ref{ellema} provides retractions $R_m:C_*(Z_{\psi _{k,l,m}}) \to C_*(M_{k,l}^m)$ for each $s_{k-1}<m\le s_n$. Each of them extends to a retraction 

$$\widetilde{R}_m : C_*(\bigcup\limits_{q=m}^{s_n} Z_{\psi _{k,l,q}} ) \to C_*(\bigcup\limits_{q=m+1}^{s_n} Z_{\psi _{k,l,q}} )  $$

When $m=s_n$, $\widetilde{R}_m$ is just another notation for $R_{s_n}$.

By Lemma \ref{ellema}, the norm of the map $R_m$ in degree $k$ is  $\|R_m\| \le (k+1)!$ for each $m$, so $\|\widetilde{R}_m\|\le (k+1)!$.

Let $c\in Z_k(M_{k,l}^{s_{k-1}})$ be the fundamental cycle of $M_{k,l}^{s_{k-1}}$. Then $\|c\|$ is the number of $k$-simplices of $M_{k,l}^{s_{k-1}}$ and $p_*(fi_{k,l})_*(c) \in Z_k(C_{k,l}^c)$ is a $k$-cycle in $C_{k,l}^c$. Thus $$ \widetilde{c}=\widetilde{R}_{s_n} \widetilde{R}_{s_n-1} \ldots \widetilde{R}_{s_{k-1}+2} \widetilde{R}_{s_{k-1}+1} p_*(fi_{k,l})_*(c) \in Z_k(M_{k,l}^{s_n})  $$ is a cycle with norm at most $((k+1)!)^{s_n-s_{k-1}}\|c\|$. But this number is exactly the number of $k$-simplices in the pseudomanifold $M_{k,l}^{s_n}$. If $\|\widetilde{c}\|< ((k+1)!)^{s_n-s_{k-1}}\|c\|$, the cycle $\widetilde{c}$ is carried by a proper subcomplex of $M_{k,l}^{s_n}$ and then it is trivial in homology. Since each $\widetilde{R}_m$ induces an isomorphism in homology, $f_*(i_{k,l})_*[M_{k,l}^{s_{k-1}}]=0\in H_k(L)$. This contradicts the assumption. Therefore, $\|\widetilde{c}\|=((k+1)!)^{s_n-s_{k-1}}\|c\|$. 

Since the equality holds, we have in particular $\| \widetilde{R}_{s_{k-1}+1} p_* (fi_{k,l})_*(c)\|= (k+1)! \|c\|$. Then for every $k$-simplex $\sigma \in M_{k,l}^{s_{k-1}}$, $S=pfi_{k,l}(\sigma)$ is a $k$-simplex of $C_{k,l}^c$ and $\| \widetilde{R}_{s_{k-1}+1}(S) \|$ $=(k+1)!$. By Lemma \ref{ellema}, $S\in M_{k,l}^{s_{k-1}}$. We conclude then that $pfi_{k,l}(M_{k,l}^{s_{k-1}})\subseteq M_{k,l}^{s_{k-1}}$ and therefore $fi_{k,l}(M_{k,l}^{s_{k-1}})\subseteq M_{k,l}^{s_{k-1}}$. If $fi_{k,l}(M_{k,l}^{s_{k-1}})$ is contained in a proper subcomplex of $M_{k,l}^{s_{k-1}}$, $(fi_{k,l})_*[M_{k,l}^{s_{k-1}}]=0$ and we have a contradiction. Then $f|_{M_{k,l}^{s_{k-1}}}: M_{k,l}^{s_{k-1}}\to M_{k,l}^{s_{k-1}}$ is an automorphism and the asymmetry of $M_{k,l}^{s_{k-1}}$ gives the desired fixed simplex. 
\end{proof}

The poset of simplices of a finite simplicial complex $K$ is denoted by $\x (K)$. Recall that a finite poset $X$ can be regarded as a topological space with finitely many points in which open sets are those subsets $U\subseteq X$ such that any $x\in X$ which is smaller than or equal to an element of $U$ is itself in $U$. This space satisfies the $T_0$ separation axiom and in fact any finite $T_0$-space is a poset in this sense. Order preserving maps correspond to continuous maps and comparable maps are homotopic (\cite{LNM}). For every finite simplicial complex $K$ there is a weak homotopy equivalence $K\to \x (K)$ (see \cite{Mcc}). The simplicial complex of chains of a poset $X$ is denoted by $\kp (X)$. There is a weak homotopy equivalence $\kp (X) \to X$. A simplicial map $\varphi :K\to L$ and a continuous map $f: X\to Y$ between finite $T_0$-spaces induce maps $\x (\varphi)$ and $\kp (f)$ in the obvious way and one has the following commutative diagrams up to homotopy where the vertical maps are the weak homotopy equivalences mentioned above
\begin{displaymath}
\xymatrix@C=18pt{ K \ar@{->}[d] \ar@{->}^{\varphi}[r] & L \ar@{->}[d] \\
									\x (K) \ar@{->}^{\x(\varphi)}[r] & \x (L)}
\hspace{2cm}									
\xymatrix@C=18pt{ X \ar@{->}[d] \ar@{->}^{f}[r] & Y \ar@{->}[d] \\
									\kp (X) \ar@{->}^{\kp (f)}[r] & \kp (Y).}									 
\end{displaymath}

\begin{coro} \label{fppsimplyconnected}
 Let $K$ be a simply connected compact CW-complex. Then there exists a topological space $X$ weak homotopy equivalent to $K$ which has the fixed point property.
 \end{coro}
\begin{proof}
 By the previous theorem there is a finite simplicial complex $L$ homotopy equivalent to $K$ with the fixed simplex property. We claim that the associated finite space $\x (L)$ has the fixed point property. Let $f:\x (L) \to \x (L)$ be a continuous map. For every $v\in L$ choose a vertex $g(v)$ of $L$ such that $g(v) \le f(v)$. The vertex map $g:L\to L$ is simplicial since $f$ maps a bounded set of minimal points into a bounded set. Then $g$ fixes some simplex $\sigma \in L$, so $\x (g):\x (L) \to \x (L)$ fixes $\sigma$. Since $f\ge \x (g)$, $f(\sigma) \ge \sigma$ and then $f^i(\sigma)=f \circ f \circ \ldots \circ f (\sigma)$ is a fixed point of $f$ for $i$ large enough.
\end{proof}

In order to extend the last corollary to non-simply connected complexes, we need to modify the construction of the space $\x (L)$. The idea we used in the proof of Theorem \ref{fixedsimplex} fails if $H_1(K)\neq 0$ since no $1$-dimensional pseudomanifold is asymmetric. We will adapt the proofs of Theorem \ref{fixedsimplex} and Corollary \ref{fppsimplyconnected} to the general case using the rigidity of finite spaces. Recall from \cite{LNM,simple} that the non-Hausdorff mapping cylinder $B_f$ of an order preserving map $f:X\to Y$ between finite $T_0$-spaces is the set $X\sqcup Y$ keeping the given ordering within $X$ and $Y$ and setting $x<y$ for $x\in X$ and $y\in Y$ if $f(x)\le y$. The cylinder $B_f$ deformation retracts to $Y$ so $\kp (B_f)$ deformation retracts to $\kp (Y)$ by a simplicial retraction. If $f$ is a weak homotopy equivalence, $\kp (B_f)$ deformation retracts to $\kp (X)$. If $X$ is a finite $T_0$-space, a point $x\in X$ such that $X_{<x}$ or $X_{>x}$ is contractible is called a \textit{weak point}. In this case the inclusion $X\smallsetminus \{x\} \hookrightarrow X$ is a weak homotopy equivalence. In other words $\kp (X)$ deformation retracts to $\kp (X\smallsetminus \{x\})$ (see \cite{LNM,simple} for more details). 

\begin{lema} \label{lemakun}
 There exists a topological space weak homotopy equivalent to $S^1$ with the fixed point property.
 \end{lema}
 \begin{proof}
 Consider the space $\kun$ of $14$ points in Figure \ref{core}. This space is the \textit{core} of a space considered by G. Kun in \cite[Remark 38]{Kun} in a different context. It is constructed by gluing two non-Hausdorff mapping cylinders of $1$ and $2$-degree maps from an $8$-point model of $S^1$ to a $4$-point model.
 
 \begin{figure}[h] 
\begin{center}
 \includegraphics[scale=0.5]{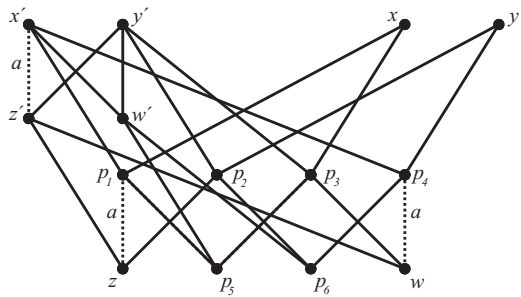}
\caption{The space $\kun$, a finite model of $S^1$ with the fixed point property.}\label{core}
\end{center}
\end{figure} 
 
Since $x$ and $y$ are weak points of $\kun$, $\kun \smallsetminus \{x,y\} \hookrightarrow \kun$ is a weak homotopy equivalence. The space $\kun \smallsetminus \{x,y\}$ is a non-Hausdorff mapping cylinder and then it deformation retracts to $\{x',y',z',w'\}$. Therefore $\kun$ is weak homotopy equivalent to $S^1$. We show that $\kun$ has the fixed point property. 
 
We will prove the following \textbf{assertion}: there is, up to sign, a unique $1$-cycle of norm at most $4$ in the order complex $\kp (\kun)$ which represents the double of a generator of $H_1(\kp (\kun))\simeq \mathbb{Z}$.

An easy way to prove this assertion is by using \textit{colorings} (see \cite{colorings}). The $\mathbb{Z}$-coloring of $\kun$ which colors the solid edges with the identity and the dotted edges with a generator $a$ of $\mathbb{Z}$ is connected and admissible, so it is the standard coloring of $\kun$. To each directed edge $vw$ of the complex $\kp (\kun)$ we assign a weight $\omega (vw)$ which is the sum of the colors of the edges in any increasing path from $v$ to $w$ if $v< w$. If $v>w$, $\omega (vw)=-\omega (wv)$. For example $\omega (zx')=a$, $\omega (w'x')=0$, $\omega (p_4w)=-a$. The map $H_1(\kp (\kun))\to \mathbb{Z}$ which maps the class of a 1-cycle $\sum v_iw_i$ to $\sum \omega (v_iw_i)$ is a well defined isomorphism (see \cite[p.208]{Spa} and \cite{colorings}). It is now easy to check that $c=zx+xw+wy+yz$ is the unique cycle of $\kp (\kun)$ with norm at most $4$ which corresponds to $2a\in \mathbb{Z}$.

Alternatively, in order to prove the assertion, the reader not familiar with colorings may consider the order complex $\kp (X)$ of the poset $X$ given by the solid edges of $\kun$. This complex is contractible and $\kp (\kun)$ is obtained from $\kp (X)$ by adding seven $1$-simplices and six $2$-simplices. Moreover, $\kp (\kun)$ collapses to $\kp (X) \cup \{x'z'\}$ and then the homology of the $1$-cycles of $\kp (\kun)$ of norm $4$ is easy to understand. 

We use now the assertion to prove the fixed point property. Suppose $f:\kun \to \kun$ is a fixed point free map. Then $\kp (f)$ has no fixed point and by the Lefschetz fixed point theorem $\kp (f)_*:H_1(\kp (\kun))\to H_1(\kp (\kun))$ is the identity. Thus $\kp (f)_*(c)=c$ and then $f$ maps $\{x,y,z,w\}$ into itself, so $f(x)=y, f(y)=x$, $f(z)=w$ and $f(w)=z$. In particular the set of points greater than $w$ and $z$ is mapped to itself, so $f(\{x',y',z'\})\subseteq \{x',y',z'\}\sqcup \{x\}\sqcup \{y\}$. If the connected subspace $\{x',y',z'\}$ is mapped into the point $x$ or into $y$, then the generating cycle $z'x'+x'w'+w'y'+y'z'$ is mapped to $0$. Therefore $\{x',y',z'\}$ is mapped into itself and then $f$ has a fixed point, a contradiction. 
 \end{proof}

  
 \begin{proof}[Proof of Theorem \ref{main}]
 We may suppose that $K$ is a finite simplicial complex. We begin with the construction of $L$ performed in the proof of Theorem \ref{fixedsimplex}, except that this time we consider also integer homology classes $\alpha _{1,1}, \alpha_{1,2}, \ldots , \alpha _{1,d_1}$ which are a basis for $H_1(K; \qq)$ and $1$-pseudomanifolds $M_{1,l}$ along with simplicial maps $\varphi _{1,l}:M_{1,l} \to K$ which map $[M_{1,l}]$ to $\alpha _{1,l}$. The pseudomanifolds $M_{k,l}$ will be assumed to be asymmetric for $k\ge 2$, but of course this is not possible for $k=1$.

 The numbers $s_1$, $N_{k,l}$, $N_k$, $s_k$ for $k\ge 2$ are defined as before. Let $s_0=0, N_1=1$ and $N=\max\limits_{1\le k\le n} N_k$. The cylinders $C_{k,l}$ are built just as before, except that $C_{k,l}^b$ will be constructed by gluing not $N$ cylinders but $(n+1)!N$ cylinders of the identity. Also, we include now the $C_{1,l}$'s. The complex $L$ is the union of all these cylinders and it is homotopy equivalent to $K$. The unique difference was the incorporation of the $1$-dimensional manifolds with their cylinders and that we increased the length of the cylinders $C_{k,l}^b$.
 
 The space $X$ will contain $\x (L)$ as a subspace. For each $1\le l\le d_1$ consider a weak homotopy equivalence $\x (M_{1,l})\to \{x',y',z',w'\}$ where the codomain is the subspace of $\kun$ defined in Lemma \ref{lemakun}. Since $\{x',y',z',w'\}\hookrightarrow \kun $ is a weak equivalence, the composition $h_l:\x (M_{l,1})\to \kun $ is a weak equivalence. We take a different copy of $\kun$ for each $1\le l\le d_1$, so the non-Hausdorff mapping cylinders $B_{h_l}$ are disjoint. Let $X=\x (L) \cup \bigcup\limits_{l=1}^{d_1} B_{h_l}$. Since $h_l$ is a weak equivalence, $\kp (B_{h_l})$ deformation retracts to $\kp (\x (M_{1,l}))=M_{1,l}'$ and then $\kp (X)$ deformation retracts to $L'$. Therefore $X$ is weak homotopy equivalent to $K$. We prove that $X$ has the fixed point property.
 
 Let $f:X\to X$ be a continuous map. Then $\kp (f):\kp (X) \to \kp (X)$ is a simplicial map. If $\kp (f)$ fixes a simplex, then $f$ fixes a chain, so $f$ fixes all the elements of the chain. For each $k\ge 1$ we consider the basis $\mathcal{B}_k=\{(i_{k,l}')_*([M_{k,l}^{s_{k-1}+1}]) \otimes 1_{\qq}\}_l$ of $H_k(\kp(X); \qq)$.  Let $k\ge 2$ and $1\le l\le d_k$. Then $\kp (f) (M_{k,l}^{s_{k-1}+1})$ lies in a complex generated by $(k+1)!\|[M_{k,l}^{s_{k-1}}]\|\le (k+1)!N_k\le (n+1)!N$ simplices. We consider now three cases in order to take into account the $1$-dimensional part: 
 
 (1) $\kp (f) (M_{k,l}^{s_{k-1}+1})\subseteq (C^{ab})'$ or
 
 (2) $\kp (f) (M_{k,l}^{s_{k-1}+1})\subseteq (C_{k',l'}^{bc})'$ for some $k'\ge 2$ and some $l'$ or
 
 (3) $\kp (f) (M_{k,l}^{s_{k-1}+1})\subseteq \kp (\x (C_{1,l'}^{bc}) \cup B_{h_{l'}})$ for some $l'$.
 
 In the third case, since $\kp (\x (C_{1,l'}^{bc}) \cup B_{h_{l'}})$ deformation retracts to $(C_{1,l'}^{bc})'$ which in turn deformation retracts into the $1$-dimensional complex $M_{1,l'}'$, then $\kp (f)_* (i_{k,l}')_*[M_{k,l}^{s_{k-1}+1}]=0\in H_k(\kp (X))$.
 
 If (1) holds, $f(\x (M_{k,l}^{s_{k-1}}))\subseteq \x (C^{ab})$. As in the proof of Corollary \ref{fppsimplyconnected}, $f: \x (M_{k,l}^{s_{k-1}})\to \x (C^{ab}) $ induces a simplicial map $g: M_{k,l}^{s_{k-1}} \to C^{ab}$ which is zero in $H_k$ by the proof of Theorem \ref{fixedsimplex}. Then $\x (g): \x (M_{k,l}^{s_{k-1}})\to \x (C^{ab})$ is zero in $k$-homology and the same is true for $f: \x (M_{k,l}^{s_{k-1}})\to \x (C^{ab}) $ since this restriction of $f$ is homotopic to $\x (g)$. Then $\kp (f): M_{k,l}^{s_{k-1}+1} \to (C^{ab})'$ induces the trivial map so $\kp (f)_* (i_{k,l}')_*[M_{k,l}^{s_{k-1}+1}]=0$. This idea works also to adapt the proof of Theorem \ref{fixedsimplex} to the cases (2) for $k'<k$, (2) for $k'>k$, (2) for $k'=k$ and $l'\neq l$. Also, if we are in the case (2) with $k'=k, l'=l$ then $f:\x(M_{k,l}^{s_{k-1}}) \to \x(C_{k,l}^{bc})$ induces $g: M_{k,l}^{s_{k-1}} \to C_{k,l}^{bc}$. If $\kp (f)_*(i_{k,l}')_*[M_{k,l}^{s_{k-1}+1}]\neq 0$, $g_*[M_{k,l}^{s_{k-1}}]\neq 0$ and then by the proof of Theorem \ref{fixedsimplex} $g$ fixes a simplex, so $\x (g)$ fixes a point $\sigma \in \x(M_{k,l}^{s_{k-1}})$. Then $f(\sigma)\ge \x(g)(\sigma)=\sigma$ and $f:X\to X$ has a fixed point. 
 
 We can then assume that the trace of $\kp (f)_*: H_k(\kp (X); \qq) \to H_k(\kp (X); \qq)$ is zero for each $k\ge 2$.
 
 Now we study the $1$-dimensional component. Let $1\le l\le d_1$. Then $\{x,y,z,w\}\subseteq B_{h_l}$ and if $j_l: \{x,y,z,w\} \hookrightarrow X$ denotes the inclusion which factors through $B_{h_l}$ and $\beta_l=[\kp (\{x,y,z,w\})]$ denotes the fundamental class of $\kp (\{x,y,z,w\})$, then $\kp (j_l)_*(\beta _l)=2(i_{1,l}')_*[M_{1,l}']\in H_1(\kp (X))$. Thus $\kp (j_l)_*(\beta _l)\otimes 1_{\qq}$ is twice an element of $\mathcal{B}_1$, the chosen basis for $H_1(\kp (X); \qq)$.
 
Since $(n+1)!N\ge 1$, for each $1\le l\le d_1$ the $l$-th copy of $\kp (\{x,y,z,w\})$ is mapped by $\kp (f)$ into $L'$ or into $\kp ( \x(C^{bc}_{1,l'}) \cup B_{h_{l'}})$ for some $l'$.  
 
 If $\kp (f)(\kp (\{x,y,z,w\}))\subseteq L'$, then it is contained in a contractible subcomplex since any closed edge-path of four edges in a barycentric subdivision is contained in the star of a vertex, and then $\kp (f)_*\kp (j_l)_*(\beta _l)=0\in H_1(\kp (X))$. If $\kp (f)(\kp (\{x,y,z,w\}))$ is contained in $\kp (\x (C_{1,l'}^{bc}) \cup B_{h_{l'}})$ for some $l'\neq l$, then $\kp (fj_l): \kp (\{x,y,z,w\}) \to \kp (\x (C_{1,l'}^{bc}) \cup B_{h_{l'}})$. The codomain of this map deformation retracts to $M_{1,l'}'$ by a retraction $r$ (which is not simplicial). Then $i_{1,l'}'r\kp (fj_l):\kp (\{x,y,z,w\}) \to \kp (X)$ is homotopic to $\kp (f)\kp(j_l)$ and therefore $\kp (f)_*\kp (j_l)_*(\beta _l)$ is an integer multiple of $(i_{1,l'}')_*[M_{1,l'}']$.
 In any of the cases considered so far, the matrix of $\kp (f)_*$ in the basis $\mathcal{B}_1$ has a zero in the entry $(l,l)$. Suppose then that $\kp (f)(\kp \{x,y,z,w\})$ is contained in $\kp (\x (C_{1,l}^{bc}) \cup B_{h_{l}})$. Since $C_{1,l}^{bc}$ deformation retracts to $M_{1,l}$ by a simplicial retraction,  $\kp (\x (C_{1,l}^{bc}) \cup B_{h_{l}})$ deformation retracts to $\kp (B_{h_l})$ by a simplicial retraction. Moreover, $\kp (B_{h_l})$ deformation retracts to $\kp (\kun)$ by a simplicial retraction which maps $M_{1,l}'$ into $\kp (\{x',y',z',w'\})$. Then $\kp (\x (C_{1,l}^{bc}) \cup B_{h_{l}})$ deformation retracts to $\kp (\kun)$ by a simplicial retraction $R:\kp (\x (C_{1,l}^{bc}) \cup B_{h_{l}}) \to \kp (\kun)$ which maps $(C_{1,l}^{bc})'$ into $\kp (\{x',y',z',w'\})$. Thus, $\kp (f)_*\kp (j_l)_*(\beta _l)=i_*R_*\kp (fj_l)_*(\beta _l)$ where $i:\kp (\kun)\to \kp (X)$ denotes the inclusion. Then the homology class $\kp (j_l)_*(\beta _l)=2 (i_{1,l}')_*[M_{1,l}']$ is mapped by $\kp (f)$ to an integer multiple of the generator $(i_{1,l}')_*[M_{1,l}']$ of the image of $i_*$. Therefore, $\kp (j_l)_*(\beta _l)$ is mapped to an integer multiple of itself. Since $f$ is a self map of a finite set, the powers of $f$ induce only finite morphisms in homology, so $\kp (j_l)_*(\beta _l)$ is mapped to $0$, $\kp (j_l)_*(\beta _l)$ or $-\kp (j_l)_*(\beta _l)$. By the Lefschetz fixed point theorem we can assume that for some $1\le l\le d_1$, $\kp (j_l)_*(\beta _l)$ is mapped to itself. Then $R_*\kp (fj_l)_*(zx+xw+wy+yz) \in Z_1(\kp (\kun))$ is a $1$-cycle of norm at most $4$ which represents the double of the generator of $H_1(\kp (\kun))$. By the assertion in Lemma \ref{lemakun} there is a unique cycle satisfying these conditions, which is $zx+xw+wy+yz$. Therefore $R \kp (fj_l)$ maps $\{x,y,z,w\}$ into itself and then $f(\{x,y,z,w\})=\{x,y,z,w\}$. Thus, the set of points smaller than one of those four points, $\{x,y,z,w,p_1,p_2, \ldots ,p_6\}$, is mapped also to itself and then $f$ maps $\kun$ into $\kun$. By Lemma \ref{lemakun}, $f$ has a fixed point.
\end{proof}




\begin{thebibliography}{99}

\bibitem{BB} K. Baclawski, A. Bj\"orner. \textit{Fixed points in partially ordered sets}.
		 Adv. Math. 31(1979), 263-287.

\bibitem{LNM} J.A. Barmak. \textit{Algebraic topology of finite topological spaces and applications}.
		Lecture Notes in Mathematics 2032, Springer, Heidelberg, 2011. xviii+170 pp.

\bibitem{simple} J.A. Barmak, E.G. Minian. \textit{Simple homotopy types and finite spaces}.
		Adv. Math. 218 (2008), Issue 1, 87-104.

\bibitem{colorings} J.A. Barmak, E.G. Minian. \textit{$G$-colorings of posets, covering maps and computation of low-dimensional homotopy groups}.
		 arXiv:1212.6442

 
\bibitem{Hat} A. Hatcher. \textit{Algebraic Topology}.
		Cambridge University Press, Cambridge, 2002. xii+544 pp.



\bibitem{Kun} G. Kun. \textit{On the fundamental group of posets}.
      Master Thesis, E\"otv\"os Lor\'and University (2003).
 
        
\bibitem{Lop} W. Lopez. \textit{An example in the fixed point theory of polyhedra}.
			Bull. Amer. Math. Soc. 73(1967), 922-924.

\bibitem{Mcc} M.C. McCord. \textit{Singular homology groups and homotopy groups of finite topological spaces}.
		Duke Math. J. 33(1966) 465-474.
        
\bibitem{Spa} E. Spanier. \textit{Algebraic Topology}.
McGraw-Hill Book Co., New York-Toronto, Ont.-London 1966 xiv+528 pp.  
      
\bibitem{Wag} R. Waggoner. \textit{A fixed point theorem for $(n-2)$-connected $n$-polyhedra}.
				Proc. Amer. Math. Soc. 33(1972), 143-145.
        
\end{thebibliography}
\end{document}